\documentclass{amsart}
\usepackage{fullpage, amssymb, mathpazo, amsmath, hyperref, color,enumitem}
\usepackage[backend=bibtex]{biblatex}
\renewbibmacro{in:}{}
\renewcommand{\a}{\mathfrak a}
\renewcommand{\b}{\mathfrak b}
\newcommand{\p}{\mathfrak p}
\newcommand{\q}{\mathfrak q}
\newcommand{\Z}{\mathbb Z}
\newcommand{\N}{\mathbb N}
\newcommand{\Q}{\mathbb Q}
\newcommand{\R}{\mathbb R}
\newcommand{\C}{\mathbb C}
\newcommand{\B}{\mathcal B}
\newcommand{\F}{\mathbb F}
\renewcommand{\to}{\rightarrow}
\renewcommand{\o}{\mathcal O}
\renewcommand{\P}{\mathcal P}
\newcommand{\m}{\mathcal M}
\newcommand{\g}{\mathcal G}
\newcommand{\h}{\mathcal H}
\renewcommand{\epsilon}{\varepsilon}
\renewcommand{\l}{\ell}
\newcommand{\Gal}{\operatorname{Gal}}

\newtheorem{thm}{Theorem}[section]
\newtheorem{lem}[thm]{Lemma}
\newtheorem{prop}[thm]{Proposition}
\newtheorem{cor}[thm]{Corollary}
\theoremstyle{definition}
\newtheorem{alg}{Algorithm}

\theoremstyle{remark}
\newtheorem*{rem}{Remark}
\newtheorem*{rems}{Remarks}

\numberwithin{equation}{section}

\title{Computing Algebraic Numbers of Bounded Height}

\author{John R. Doyle}
\address{Department of Mathematics \\
University of Georgia \\
Athens, GA  30602} 
\email{jdoyle@math.uga.edu, dkrumm@math.uga.edu}

\author{David Krumm}

\begin{document}
\begin{abstract}We describe an algorithm which, given a number field $K$ and a bound $B$, finds all the elements of $K$ having relative height at most $B$. Two lists of numbers are computed: one consisting of elements $x\in K$ for which it is known with certainty that $H_K(x)\leq B$, and one containing elements $x$ such that $|H_K(x)-B|<\theta$ for a tolerance $\theta$ chosen by the user. We show that every element of $K$ whose height is at most $B$ must appear in one of the two lists.
\end{abstract}
 
\maketitle

\section{Introduction} Let $K$ be a number field with ring of integers $\o_K$, and let $H_K$ be the relative height function on $K$. For any bound $B$ it is known that the set of all elements $x\in K$ with $H_K(x)\leq B$ is finite \cite[\S3.1]{silverman}. Moreover, there is an asymptotic formula for the number of such elements, due to Schanuel \cite{schanuel}: \[\#\{x\in K:H_K(x)\leq B\}\sim C_KB^2,\] where $C_K$ is an explicit constant which depends only on $K$. However, there does not appear to be in the literature an algorithm that would allow fast computation of all these elements. In \cite{petho/schmitt} Peth\H{o} and Schmitt require such an algorithm to be able to compute the Mordell-Weil groups of certain elliptic curves over real quadratic fields. They obtain an algorithm by showing that if $\omega_1,\ldots, \omega_n$ is an LLL-reduced integral basis of $K$, then every element $x\in K$ with $H_K(x)\leq B$ can be written as \[x=\frac{a_1\omega_1+\cdots+a_n\omega_n}{c},\] where $a_1,\ldots, a_n,c$ are integers within certain explicit bounds depending only on $B$ and $K$. The set of all such numbers $x$ is finite, so one would only need to search through this set and discard elements whose height is greater than $B$. Unfortunately, in practice this method is slow because the search space is very large. We describe in this paper  an algorithm which is faster, assuming class group representatives for $\o_K$ and a basis for the unit group of $\o_K$ can be computed efficiently. Sample computations showing the greater speed of our method may be seen in \S\ref{comp_perf}.

Our motivation for designing a fast algorithm that can handle relatively large bounds $B$ comes from arithmetic dynamics. In \cite{poonen} Poonen provides a conjecturally complete list of rational preperiodic graph structures for quadratic polynomial maps defined over $\Q$. It is then natural to ask which preperiodic graph structures can occur for such maps over other number fields. In order to gather data about these graphs one needs to be able to compute all the preperiodic points of a given quadratic polynomial. It is possible to give an explicit  upper bound for the height of any preperiodic point of a given map, so a first step towards computing preperiodic points is computing all points of bounded height. Further details on this question, together with the generated data, will be presented in a subsequent paper \cite{jxd1}.

{\bf Acknowledgements.} We thank Xander Faber, Dino Lorenzini, and Andrew Sutherland for their comments on earlier versions of this paper; Pete Clark for help finding references; Robert Rumely for suggestions on improving the efficiency of the algorithm; and the referee for several helpful comments.
  
 \section{Background and notation}
 
Let $K$ be a number field; let $\sigma_1,\ldots, \sigma_{r_1}$ be the real embeddings of $K$, and $\tau_1,\overline\tau_1,\ldots, \tau_{r_2},\overline\tau_{r_2}$ the complex embeddings. Corresponding to each of these embeddings there is an archimedean absolute value on $K$ extending the usual absolute value on $\Q$. For an embedding $\sigma$, the corresponding absolute value $|\;|_{\sigma}$ is given by $|x|_{\sigma}=|\sigma(x)|_{\C}$, where $|\cdot |_{\C}$ is the usual complex absolute value. Note that $|\;|_{\overline\tau_i}=|\;|_{\tau_i}$ for every $i$. We will denote by $M_K^{\infty}$ the set of absolute values corresponding to $\sigma_1,\ldots,\sigma_{r_1},\tau_1,\ldots, \tau_{r_2}$.

For every maximal ideal $\p$ of the ring of integers $\o_K$ there is a discrete valuation $v_{\p}$ on $K$ with the property that for every $a\in K^{\ast}$, $v_{\p}(a)$ is the power of $\p$ dividing the principal ideal $(a)$. If $\p$ lies over the prime $p$ of $\Z$, there is an absolute value $|\;|_{\p}$ on $K$ extending the $p$-adic absolute value on $\Q$. Let $e(\p)$ and $f(\p)$ denote the ramification index and residual degree of $\p$, respectively. This absolute value is then given by $|x|_{\p}=(N\p)^{-v_{\p}(x)/(e(\p)f(\p))}$. We denote by $M_K^0$ the set of all absolute values $|\;|_{\p}$, and we let $M_K=M_K^{\infty}\cup M_K^0$.

For an absolute value $v\in M_K$, let $K_v$ be the completion of $K$ with respect to $v$, and let $\Q_v$ be the completion of $\Q$ with respect to the restriction of $v$ to $\Q$. Note that $\Q_v=\R$ if $v\in M_K^{\infty}$, and $\Q_v=\Q_p$ if $v\in M_K^0$ corresponds to a maximal ideal $\p$ lying over $p$. The {\it local degree} of $K$ at $v$ is given by $n_v=[K_v:\Q_v]$. If $v$ corresponds to a real embedding of $K$, then $K_v=\Q_v=\R$, so $n_v=1$. If $v$ corresponds to a complex embedding of $K$, then $K_v=\C$ and $\Q_v=\R$, so $n_v=2$. Finally, if $v$ corresponds to a maximal ideal $\p$, then $n_v=e(\p)f(\p)$.

The {\it relative height} function $H_K:K\to\R_{\geq 1}$ is defined by \[H_K(\gamma)=\prod_{v\in M_K}\max\{|\gamma|_v^{n_v},1\}\] and has the following properties:

\begin{itemize}[itemsep = 1.1mm]
\item For any $\alpha,\beta \in K$ with $\beta\neq 0$, $H_K(\alpha/\beta)=\prod_{v\in M_K}\max\{|\alpha|_v^{n_v},|\beta|_v^{n_v}\}.$
\item For any $\alpha,\beta\in K$, $H_K(\alpha\beta)\leq H_K(\alpha)H_K(\beta).$ 
\item For any $\alpha,\beta\in\o_K$ with $\beta\neq 0$, $H_K(\alpha/\beta)=N(\alpha,\beta)^{-1}\prod_{v\in M_K^{\infty}}\max\{|\alpha|_v^{n_v},|\beta|_v^{n_v}\}.$ \\Here $N(\alpha,\beta)$ denotes the norm of the ideal generated by $\alpha$ and $\beta$.
\item For any $\gamma\in K^{\ast}$, $H_K(\gamma)=H_K(1/\gamma)$.
\item For any $\gamma\in K$ and any root of unity $\zeta\in K$, $H_K(\zeta\gamma)=H_K(\gamma)$.
\end{itemize}

It will sometimes be convenient to use the logarithmic height function $h_K=\log\circ H_K$.

The following notation will be used throughout: $\o_K^{\times}$ is the unit group of $\o_K$, $\mu_K$ is the group of roots of unity in $K$, $r=r_1+r_2-1$ is the rank of $\o_K^{\times}$, $h$ is the class number of $K$, and $\Delta_K$ is the discriminant of $K$. For an ideal $I$ of $\o_K$ we let $N(I):=\#(\o_K/I)$ denote the norm of the ideal.

Define a logarithmic map $\Lambda: K^{\ast}\to\R^{r+1}$ by \[\Lambda(x)=(\log |x|_v^{n_v})_{v\in M_K^{\infty}}=\left(\log |x|_{\sigma_1},\ldots,\log |x|_{\sigma_{r_1}},\log |x|_{\tau_1}^2,\ldots,\log |x|_{\tau_{r_2}}^2\right).\]  Note that $\Lambda$ is a group homomorphism. By a classical result of Kronecker, the kernel of $\Lambda$ is $\mu_K$. Letting $\pi:\R^{r+1}\to \R^r$ be the projection map that deletes the last coordinate, we set $\Lambda'=\pi\circ\Lambda$.

Recall that there is a system $\boldsymbol\epsilon=\{\epsilon_1,\ldots,\epsilon_r\}\subset\o_K^{\times}$ of {\it fundamental units} such that every unit $u\in\o_K^{\times}$ can be written uniquely as $u=\zeta\epsilon_1^{n_1}\cdots\epsilon_r^{n_r}$ for some integers $n_1,\ldots,n_r$ and some $\zeta\in \mu_K$. We denote by $S(\boldsymbol\epsilon)$ the $r\times r$ matrix with column vectors $\Lambda'(\epsilon_j)$.
 
\section{The method}

\noindent Let $K$ be a number field, and let $H_K:K\to\R_{\geq 1}$ be the relative height function on $K$. Given a bound $B\geq 1$, we want to list the elements $\gamma\in K$ satisfying $H_K(\gamma)\leq B$. Our method for finding all such numbers is based on the observation that, using the ideal class group and unit group of $\o_K$, this problem can be reduced to the question of finding all units of bounded height. In essence, the idea is to generalize the following statement that holds over $\Q$: if $x\in\Q^{\ast}$ and $H_{\Q}(x)\leq B$, then $x$ can be written as $x=\pm a/b$ where $a$ and $b$ are integers such that $(a,b)=1$ and $|a|,|b|\leq B$. For a general number field $K$, the analogous statement we make is that given $x\in K^{\ast}$ with $H_K(x)\leq B$, it is possible to write $x$ in the form $x=u\cdot a/b$, where $u\in\o_K^{\times}$ is a unit whose height is explicitly bounded; $a$ and $b$ are elements of $\o_K$ such that $(a,b)=\a$, where $\a$ is one ideal from a predetermined list of ideal class representatives for $\o_K$; and $|N_{K/\Q}(a)|, |N_{K/\Q}(b)|\leq B\cdot N(\a)$.

\subsection{The main algorithm}\label{main_thm_section} Theorem \ref{main_thm} below provides the theoretical basis for our algorithm. In order to describe all elements of bounded height in $K$ we fix integral ideals $\a_1,\ldots,\a_h$ forming a complete set of ideal class representatives for $\o_K$, and we fix fundamental units $\epsilon_1,\ldots,\epsilon_r$. Suppose we are given a bound $B\geq 1$. For each ideal $\a_{\l}$, let  $g_{\l,1},\ldots, g_{\l,s_{\l}}$ be generators for all the nonzero principal ideals contained in $\a_{\l}$ whose norms are at most $B\cdot N(\a_{\l})$. We define a $B$-{\it packet} to be a tuple of the form \[P=\left(\l,(i,j),(n_1,\ldots, n_r)\right)\] satisfying the following conditions: 
\begin{itemize}
\item $1\leq\l\leq h$; 
\item $1\leq i<j\leq s_{\l}$;
\item $(g_{\l,i},g_{\l,j})=\a_{\l}$; and
\item $H_K(\epsilon_1^{n_1}\cdots \epsilon_r^{n_r})\leq B\cdot H_K(g_{\l,i}/g_{\l,j})$. 
\end{itemize}

To a packet $P$ we associate the number \[c(P)=\epsilon_1^{n_1}\cdots \epsilon_r^{n_r}\cdot\frac{g_{\l,i}}{g_{\l,j}}\in K^{\ast}\backslash\o_K^{\times}\] and the set \[F(P)=\{\zeta \cdot c(P) :\zeta\in\mu_K\}\cup\{\zeta/c(P) :\zeta\in\mu_K\}.\] 

Note that the union defining $F(P)$ is disjoint, and that $F(P)$ does not contain units. Moreover, all the elements of $F(P)$ have the same height. If $r=0$, then a packet is a tuple of the form $(\l,(i,j))$ satisfying only the first three defining conditions above, and in this case $c(P)=g_{\l,i}/g_{\l,j}$.

With the notation and terminology introduced above we can now describe all elements of $K$ whose height is at most $B$.

\begin{thm}\label{main_thm} Suppose that $\gamma\in K^{\ast}$ satisfies $H_K(\gamma)\leq B$. Then either $\gamma\in\o_K^{\times}$ or $\gamma$ belongs to the disjoint union \[\bigcup_{B\text{-packets\;} P} F(P).\]
\end{thm}

\begin{proof} Assuming that $\gamma\notin\o_K^{\times}$ we must show that there is a packet $P$ such that $\gamma\in F(P)$. We can write the fractional ideal generated by $\gamma$ as $(\gamma)=IJ^{-1}$, where $I$ and $J$ are coprime integral ideals. Since $I$ and $J$ are in the same ideal class, there is some ideal $\a_{\l}$ (namely the one representing the inverse class of $I$ and $J$) such that $\a_{\l}I$ and $\a_{\l}J$ are principal; say $(\alpha)=\a_{\l}I,(\beta)=\a_{\l}J$. Note that $(\alpha,\beta)=\a_{\l}$ because $I$ and $J$ are coprime. Since $(\gamma)=(\alpha)(\beta)^{-1}$ we may assume, after scaling $\alpha$ by a unit, that $\gamma=\alpha/\beta$. From the bound $H_K(\gamma)\leq B$ it follows that \[\prod_{v\in M_K^{\infty}}\max\{|\alpha|_v^{n_v},|\beta|_v^{n_v}\}\leq B\cdot N(\a_{\l}).\] In particular, \[|N_{K/\Q}(\alpha)|=\prod_{v\in M_K^{\infty}}|\alpha|_v^{n_v}  \leq B\cdot N(\a_{\l})\hspace{3mm}\mathrm{and}\hspace{3mm}|N_{K/\Q}(\beta)|=\prod_{v\in M_K^{\infty}}|\beta|_v^{n_v}  \leq B\cdot N(\a_{\l}).\]  Since $N(\alpha), N(\beta)\leq B\cdot N(\a_{\l})$, there must be some indices $a,b\leq s_{\l}$ such that $(\alpha)=(g_{\l,a})$ and $(\beta)=(g_{\l,b})$. Hence, we have $\alpha=g_{\l,a}u_a$ and $\beta=g_{\l,b}u_b$ for some units $u_{a},u_{b}$. Letting $t=u_a/u_b$ we have $\gamma=tg_{\l,a}/g_{\l,b}$, and since $H_K(\gamma)\leq B$, then \[H_K(t)=H_K(\gamma g_{\l,b}/g_{\l,a})\leq H_K(\gamma)H_K(g_{\l,b}/g_{\l,a})\leq B\cdot H_K(g_{\l,b}/g_{\l,a}).\] 

Write $t=\zeta\epsilon_1^{m_1}\cdots\epsilon_r^{m_r}$ for some integers $m_1,\ldots,m_r$ and some $\zeta\in \mu_K$. We define indices $i,j$ and an integer tuple $(n_1,\ldots, n_r)$ as follows: if $a< b$, we let $i=a, j=b, (n_1,\ldots, n_r)=(m_1,\ldots, m_r)$; and if $a>b$, we let $i=b, j=a,  (n_1,\ldots, n_r)=(-m_1,\ldots, -m_r)$. (The case $a=b$ cannot occur since $\gamma$ is not a unit.) Note that in either case we have $i< j$ and $(g_{\l,i},g_{\l,j})=(\alpha,\beta)=\a_{\l}$. Letting $u= \epsilon_1^{n_1}\cdots\epsilon_r^{n_r}$ we have $H_K(u)=H_K(t)$, so $H_K(u)\leq B\cdot H_K(g_{\l,i}/g_{\l,j})$. This proves that $P:=\left(\l,(i,j),(n_1,\ldots, n_r)\right)$ is a $B$-packet. Finally, if we set $c=ug_{\l,i}/g_{\l,j}$, then $\zeta c=\gamma$ if $a< b$; and $\zeta/c=\gamma$ if $a>b$.  Therefore, $\gamma\in F(P)$.

We show now that the union in the statement of the theorem is disjoint. Suppose that \[P=\left(\l,(i,j),(n_1,\ldots, n_r)\right)\; \mathrm{\;and\;} \;P'=\left(\l',(i',j'),(n_1',\ldots, n_r')\right)\] are packets such that $F(P)\cap F(P')\neq\emptyset$. We aim to show that $P=P'$. Let $u=\epsilon_1^{n_1}\cdots \epsilon_r^{n_r}$, and similarly define $u'$. From the assumption that $F(P)$ and $F(P')$ have a common element it follows that either \[c(P)\cdot c(P')\in\mu_K \;\mathrm{\;or\;} \;c(P)/c(P')\in\mu_K.\] We consider the latter case first. There are ideals $\b_{\l,i}, \b_{\l,j}, \b_{\l',i'}, \b_{\l',j'}$ such that
\begin{equation}\label{g_ideals}(g_{\l,i})=\a_{\l}\b_{\l,i}\;;\; (g_{\l,j})=\a_{\l}\b_{\l,j}\;;\;\  (g_{\l',i'})=\a_{\l'}\b_{\l',i'}\;;\;  (g_{\l',j'})=\a_{\l'}\b_{\l',j'}\;.
\end{equation} 
Note that $\b_{\l,i}$ and $\b_{\l,j}$ are coprime because $(g_{\l,i},g_{\l,j})=\a_{\l}$; similarly, $\b_{\l',i'}$ and $\b_{\l',j'}$ are coprime. Now, since $c(P)/c(P')\in\mu_K$, there is an equality of ideals $(g_{\l,i})(g_{\l',j'})=(g_{\l,j})(g_{\l',i'})$. Therefore, $\b_{\l,i}\b_{\l',j'}=\b_{\l,j}\b_{\l',i'}$ and by coprimality we conclude that 
\begin{equation}\label{b_ideals}
\b_{\l,i}=\b_{\l',i'}\; \mathrm{\;and\;} \;\b_{\l,j}=\b_{\l',j'}.
\end{equation}
 Considering ideal classes, by \eqref{g_ideals} and \eqref{b_ideals} we obtain \[[\a_{\l}]^{-1}=[\b_{\l,i}]=[\b_{\l',i'}]=[\a_{\l'}]^{-1},\] so $\l=\l'$. Thus, again using \eqref{g_ideals} and \eqref{b_ideals}, \[(g_{\l,i})=\a_{\l}\b_{\l,i}=\a_{\l'}\b_{\l',i'}=(g_{\l',i'})=(g_{\l,i'}),\] and hence $i=i'$; similarly, $j=j'$. It follows that $u/u'=c(P)/c(P')\in\mu_K$, so $(n_1,\ldots, n_r)=(n_1',\ldots, n_r')$, and therefore $P=P'$.

The case where $c(P)\cdot c(P')\in\mu_K$ is dealt with similarly, and leads to the conclusion that $(i,j)=(j',i')$. But this is a contradiction, since $i<j$ and $i'<j'$; therefore, this case cannot occur.
\end{proof}

\begin{rem} In the case where $r=0$, Theorem \ref{main_thm} and its proof still hold if we omit mention of the fundamental units. See \textsection\ref{bdd_height_quad} for a refinement of the theorem in this case.
\end{rem}

From Theorem \ref{main_thm} we deduce the following algorithm.

\begin{alg}[Algebraic numbers of bounded height]\label{main1} \mbox{}\\
Input: A number field $K$ and a bound $B\geq 1$.\\
Output: A list of all elements $x\in K$ satisfying $H_K(x)\leq B$.
\begin{enumerate}[itemsep = 1.1mm]
\item Create a list $L$ containing only the element 0.
\item Determine a complete set $\{\a_1,\ldots,\a_h\}$ of ideal class representatives for $\o_K$.
\item Compute a system $\boldsymbol\epsilon=\{\epsilon_1,\ldots,\epsilon_r\}$ of fundamental units.
\item Include in $L$ all units $u\in\o_K^{\times}$ with $H_K(u)\leq B$.
\item For each ideal $\a_{\l}$\;:
\begin{enumerate}
\item Find generators $g_{\l,1},\ldots, g_{\l,s_{\l}}$ for all the nonzero principal ideals contained in $\a_{\l}$ whose norms are at most $B\cdot N(\a_{\l})$.
\item For each pair of indices $i,j$ such that $1\leq i< j\leq s_{\l}$ and $(g_{\l,i},g_{\l,j})=\a_{\l}$\;: 
\begin{enumerate}
\item Find all units $u$ of the form $u=\epsilon_1^{n_1}\cdots\epsilon_r^{n_r}$ such that $H_K(u)\leq B\cdot H_K(g_{\l,i}/g_{\l,j})$.
\item For all such units $u$, let $c=u\cdot g_{\l,i}/g_{\l,j}$. If $H_K(c)\leq B$, then append to $L$ all elements of the form $\zeta\cdot c$ and $\zeta/c$ with $\zeta\in\mu_K$.
\end{enumerate}
\end{enumerate}
\item Return the list $L$.
\end{enumerate} 
\end{alg}

Note that, by Theorem \ref{main_thm}, the list $L$ will not contain duplicate elements. There are known methods for carrying out steps (2) and (3) of Algorithm \ref{main1} --- see, for instance, \cite[\S6.5]{cohen}. To do step 5(a) one can use methods for finding elements in $\o_K$ of given norm; one algorithm for doing this can be found in \cite{fincke/pohst}. It remains to explain how a set of units of bounded height can be computed.

\subsection{Units of bounded height}\label{units} For a given bound $D\geq 1$ we wish to determine all units $u\in\o_K^{\times}$ such that $H_K(u)\leq D$. Our method for doing this makes use of the following classical result.

\begin{thm}[Dirichlet] The map $\Lambda':\o_K^{\times}\to\R^r$ is a group homomorphism with kernel $\mu_K$,  and $\Lambda'(\o_K^{\times})$ is a lattice of full rank in $\R^r$ spanned by the vectors $\Lambda'(\epsilon_1),\ldots,\Lambda'(\epsilon_r)$.
\end{thm}

Let $S=S(\boldsymbol\epsilon)$ be the $r\times r$ matrix with column vectors $\Lambda'(\epsilon_j)$, and let $T=S^{-1}$ be the linear automorphism of $\R^r$ taking the basis $\Lambda'(\epsilon_1),\ldots,\Lambda'(\epsilon_r)$ to the standard basis for $\R^r$.

\begin{prop}\label{bdd_units} Suppose $u\in\o_K^{\times}$ satisfies $H_K(u)\leq D$. Then there exist an integer point $(n_1,\ldots, n_r)$ in the polytope $T([-\log D,\log D]^r)$ and a root of unity $\zeta\in\mu_K$ such that $u=\zeta\epsilon_1^{n_1}\cdots \epsilon_r^{n_r}$.
\end{prop}

\begin{proof} The bound $H_K(u)\leq D$ implies that $|u|_v^{n_v}\leq D$ for all $v\in M_K^{\infty}$. Since $H_K(1/u)=H_K(u)$ we also have $1/|u|_{v}^{n_v}\leq D$. Therefore, \[-\log D\leq \log |u|_v^{n_v}\leq\log D \hspace{3mm}\mathrm{for\;all\;} v\in M_K^{\infty},\] 

so $\Lambda'(u)\in[-\log D,\log D]^r$. We can write $u=\zeta\epsilon_1^{n_1}\cdots \epsilon_r^{n_r}$ for some $\zeta\in\mu_K$ and some integers $n_i$. Then $(n_1,\ldots, n_r)=T(\Lambda'(u))\in T([-\log D,\log D]^r)$.
\end{proof}

The above proposition leads to the following algorithm.

\begin{alg}[Units of bounded height]\label{units1}\mbox{}\\
Input: A number field $K$ and a bound $D\geq 1$.\\
Output: A list of all units $u\in \o_K^{\times}$ satisfying $H_K(u)\leq D$.
\begin{enumerate}[itemsep = 1.1mm]
\item If $r=0$, return $\mu_K$. Otherwise:
\item Create an empty list $U$.
\item Compute fundamental units $\epsilon_1,\ldots, \epsilon_r$.
\item Find all integer points $Q$ in the polytope $T([-\log D,\log D]^r)$.
\item For all such points $Q=(n_1,\ldots, n_r)$:
\begin{enumerate}
\item Let $u=\epsilon_1^{n_1}\cdots \epsilon_r^{n_r}$. 
\item If $H_K(u)\leq D$, then include $u\zeta$ in $U$ for all $\zeta\in\mu_K.$
\end{enumerate}
\item Return the list $U$.
\end{enumerate}
\end{alg}

Step (4) of Algorithm \ref{units1} can be done using known methods for finding integer points in polytopes; see the articles \cite{barvinok/pommersheim,LattE}. 

\begin{rem} With more work it is possible to replace the box $[-\log D,\log D]^r$ in step 4 of Algorithm \ref{units1} with a substantially smaller set, namely the polytope $\mathcal P(D)$ in $\R^r$ cut out by the inequalities 
\[ -\log D \le \sum_{i \in I} x_i \leq \log D, \] where $I$ runs through all nonempty subsets of $\{1,\ldots,r\}$. This polytope is contained in the box $[-\log D,\log D]^r$, and one can show that its volume is smaller than that of the box by a factor of at least $(\lfloor r/2\rfloor !)^2$. In addition to providing a smaller search space, using $\mathcal P(D)$ eliminates the need to check the heights of the units obtained. This is due to the fact that for units $u$, $H_K(u)\leq D$ if and only if $\Lambda'(u)\in\mathcal P(D)$. We omit the proofs of these statements since we will not use the polytope $\mathcal P(D)$ here; the box $[-\log D, \log D]^r$ works well in practice and will suffice for a theoretical analysis of the main algorithm. 
\end{rem}

For later reference we record the following facts concerning units of bounded height.

\begin{lem}\label{unit_tuple_bound} If the unit $u=\zeta\epsilon_1^{n_1}\cdots \epsilon_r^{n_r}$ satisfies $H_K(u)\le D$, then \[\max_{1\leq i\leq r}|n_i|\le M:=\lfloor\|T\|\cdot\sqrt r\cdot \log D\rfloor,\] where $\|T\|$ denotes the operator norm of $T$.
\end{lem}

\begin{proof} By Proposition \ref{bdd_units}, the point $(n_1,\ldots, n_r)$ belongs to $T([-\log D,\log D]^r)$, and therefore has Euclidean norm bounded by $\|T\|\cdot\sqrt r\cdot \log D$. Since each $n_i$ is an integer, it follows that $|n_i|\le M$ for all $i$.
\end{proof}

Using Lemma \ref{unit_tuple_bound} we can deduce upper bounds for the number of units of bounded height.

\begin{cor}\label{unit_number} Fix $\kappa>0$. There is a constant $q=q(\kappa, K,\boldsymbol\epsilon)$ such that for every bound $D\geq 1+\kappa$, the number of units $u\in\o_K^{\times}$ satisfying $H_K(u)\leq D$ is at most $q\cdot(\log D)^r$.
\end{cor}

\begin{proof} Suppose $u=\zeta\epsilon_1^{n_1}\cdots \epsilon_r^{n_r}$ is a unit with $H_K(u)\leq D$. By Lemma \ref{unit_tuple_bound}, $(n_1,\ldots, n_r)\in [-M,M]^r$. This gives at most $(2M+1)^r$ options for the tuple $(n_1,\ldots, n_r)$, and hence at most $(\#\mu_K)\cdot(2M+1)^r$ options for $u$. Therefore, \[\frac{\#\{u\in\o_K^{\times}:H_K(u)\leq D\}}{(\log D)^r}\leq(\#\mu_K)\left(2\|T\|\cdot\sqrt{r}+\frac{1}{\log (1+\kappa)}\right)^r,\] and the result follows.
\end{proof}

While Algorithms \ref{main1} and \ref{units1} form a theoretically accurate description of our method, for purposes of explicit computation they are not optimal. We discuss now a few changes to our method which will make it more efficient.

\subsection{Computational improvements to the method}\label{improvements} We aim in this section to modify Algorithms \ref{main1} and \ref{units1} with the following goals in mind: to avoid computing any given piece of data more than once; to minimize the cost of height computations; and to avoid, as much as possible, doing arithmetic with fundamental units. The latter is desirable because fundamental units in a number field can be very large, so that arithmetic operations with them might be costly.

Regarding the expense of height computations, we begin by noting that the height of an element of $K$ can be computed by using the logarithmic map $\Lambda$. Indeed, suppose $\alpha,\beta$ are nonzero elements of $\o_K$; letting $\Lambda(\alpha)=(x_1,\ldots, x_{r+1})$ and $\Lambda(\beta)=(y_1,\ldots, y_{r+1})$ we have
\begin{equation}\label{height_via_log} \log(N(\alpha,\beta))+ h_K(\alpha/\beta)=\sum_{i=1}^{r+1}\max\{x_i,y_i\}.
\end{equation}

In view of this fact, throughout this section we will use the logarithmic height function $h_K$ rather than $H_K$. From a computational standpoint, $h_K$ is also more convenient because it is defined as a sum rather than a product.

To minimize the amount of time spent on height computations in step (5) of Algorithm \ref{main1}, we make the following observation: using \eqref{height_via_log}, all of the heights required in that step can be computed from the data of the real vectors $\Lambda(\epsilon_1),\ldots, \Lambda(\epsilon_r)$ and the vectors $\Lambda(g_{\l,i})$ for all indices $\l, i$: 
\begin{itemize}[itemsep = 1.1mm] 
\item The height of a unit $u=\epsilon_1^{n_1}\cdots \epsilon_r^{n_r}$ can be found knowing only the tuple $(n_1,\ldots, n_r)$ --- without actually computing $u$. Indeed, $h_K(u)$ can be computed from the vector $\Lambda(u)$, which is equal to $\sum_{j=1}^rn_j\Lambda(\epsilon_j)$. 
\item The numbers $h_K(g_{\l,i}/g_{\l,j})$ required in step 5(b)(i) of Algorithm \ref{main1} can be computed from the vectors $\Lambda(g_{\l,i})$ and $\Lambda(g_{\l,j})$.
\item Using the fact that $\Lambda(u\cdot g_{\l,i})=\Lambda(u)+\Lambda(g_{\l,i})$, the number $h_K(u\cdot g_{\l,i}/g_{\l,j})$ in step 5(b)(ii) can be computed from the tuple $(n_1,\ldots, n_r)$ and the vectors $\Lambda(g_{\l,i})$, $\Lambda(g_{\l,j})$.
\end{itemize}

From these observations we conclude that the vectors $\Lambda(\epsilon_1),\ldots, \Lambda(\epsilon_r)$ and $\Lambda(g_{\l,i})$ (for all appropriate indices $\l,i$) should be computed once and stored for later use; all height computations that take place within the algorithm can then make use of this precomputed data. 

Step 5(b)(i) of Algorithm \ref{main1}, in which we compute all units of height less than a given bound, must be performed many times --- each time with a different height bound. It would be more efficient to let $d$ be the largest height bound considered, and determine the list $U$ of units $u$ satisfying $h_K(u)\leq d$. This list will then contain all units needed throughout the algorithm. In particular, the units from step (4) can be obtained from $U$. Hence, step (4) should be carried out only after the list $U$ has been computed. By making these changes, only one computation of units of bounded height will be required throughout the entire algorithm.

Finally, in order to speed up computations involving the vectors $\Lambda'(\epsilon_1),\ldots,\Lambda'(\epsilon_r)$,  it will be useful to assume that these vectors form a suitably reduced basis for the lattice $\Lambda'(\o_K^{\times})\subset\R^r$. We will therefore apply the LLL algorithm \cite{lll} to obtain a reduced basis.

With all of the above modifications in mind we now give an improved version of our method.

\begin{alg}[Algebraic numbers of bounded height]\label{main2}\mbox{}\\
Input: A number field $K$ and a bound $B\geq 1$.\\
Output: A list $L$ of all elements $\gamma \in K$ satisfying $H_K(\gamma) \le B$.

\smallskip
\begin{enumerate}[itemsep = 1.1mm]
\item Find a complete set $\{\a_1,\ldots,\a_h\}$ of ideal class representatives for $\o_K$ and compute an LLL-reduced basis $\boldsymbol\epsilon=\{\epsilon_1,\ldots,\epsilon_r\}$ of fundamental units in $\o_K$. Compute the numbers $\log N(\a_{\l})$ for every index $\l$, and the vectors $\Lambda(\epsilon_j)$ for every $j$.
\item Construct the set \[\mathcal N:=\bigcup_{\l=1}^h\{m\cdot N(\a_{\l}):m\le B\}\] and make a list $\mathcal{P}$ of all nonzero principal ideals of $\o_K$, each represented by a single generator $g$, having norm in $\mathcal N$. Record $\Lambda(g)$ for each $g$.
\item For each ideal $\a_{\l}$, make a list $(g_{\l,1}),\ldots, (g_{\l,s_{\l}})$ of all elements of $\mathcal{P}$ contained in $\a_{\l}$ whose norms are at most $B\cdot N(\a_{\l})$.
\item For each index $\l$:
\begin{enumerate}
\item Make a list $R_{\l}$ of pairs $(i,j)$ such that $1\leq i < j \leq s_{\l}$ and $(g_{\l,i},g_{\l,j})=\a_{\l}$. 
\item For each pair $(i,j)$ in $R_{\l}$, use the data recorded in step (2) to compute $h_{\l,i,j}=h_K(g_{\l,i}/g_{\l,j})$.
\end{enumerate}
\item Let $d= \log B + \max_{\l}\max_{(i,j)\in R_{\l}} h_{\l,i,j}.$
\item Construct the matrix $S=S(\boldsymbol\epsilon)$ with column vectors $\Lambda'(\epsilon_1),\ldots, \Lambda'(\epsilon_r)$ and compute its inverse.
\item Construct a list $U$ consisting of all integer vectors $(n_1,\ldots,n_r)$ in the polytope $S^{-1}([-d,d]^r)$.
\item Create a list $L$ containing only the element 0, and create empty lists $U_0$ and $L_0$.
\item For each tuple $\boldsymbol u=(n_1,\ldots,n_r)$ in $U$:
	\begin{enumerate}
	\item Compute the vector $\Lambda_{\boldsymbol u}=\sum_{j=1}^rn_j\Lambda(\epsilon_j)$ using the data from step (1).
	\item Use $\Lambda_{\boldsymbol u}$ to compute $h_{\boldsymbol u}=h_K\left(\epsilon_1^{n_1} \cdots \epsilon_r^{n_r}\right)$.
	\item If $h_{\boldsymbol u} \le \log B$, then append $\boldsymbol u$ to $U_0$.
	\item If $h_{\boldsymbol u} > d$, then remove $\boldsymbol u$ from $U$.
	\end{enumerate}
	
\item For each index $\l$\;: 
\begin{itemize}[itemindent = -2mm]
\item[]For each pair $(i,j)\in R_{\l}$\;:
\begin{itemize}[itemindent = -5mm]
\item[] Let $w = \log B + h_{\l,i,j}$.
\item[] For each tuple $\boldsymbol u=(n_1,\ldots,n_r)$ in $U$:
\begin{enumerate}[itemindent = -6mm]
\item[] If $h_{\boldsymbol u}\leq w$, then:
\begin{enumerate}[itemindent = 1mm]
\item Let $P$ be the packet $\left(\l,(i,j),(n_1,\ldots,n_r)\right)$.
\item Use the data recorded in steps (2) and 9(a) to compute $h_K(c(P))$.
\item If $h_K(c(P)) \le \log B$, then append the packet $P$ to $L_0$.
\end{enumerate}
\end{enumerate}
\end{itemize}
\end{itemize}
\item Make a list consisting of the distinct tuples $(n_1,\ldots, n_r)$ appearing in $U_0$ or in some packet $P\in L_0$. Compute and store all units of the form $\epsilon_1^{n_1} \cdots \epsilon_r^{n_r}$ with $(n_1,\ldots, n_r)$ in this list.
\item For each tuple $(n_1,\ldots,n_r)$ in $U_0$, append to $L$ all numbers of the form $\zeta \epsilon_1^{n_1} \cdots \epsilon_r^{n_r}$ with $\zeta\in\mu_K$. Use the units computed in step (11).
\item For each packet $P$ in $L_0$, compute the number $c(P)$ using the data from step (11), and append to $L$ all numbers of the form $\zeta\cdot c(P)$ and $\zeta/c(P)$ with $\zeta\in\mu_K$.
\item Return the list $L$.
\end{enumerate}
\end{alg}

\begin{rems}\mbox{}
\begin{itemize}[itemsep = 1.1mm]
\item Step (11) is done in order to avoid computing the same unit more than once: it may well happen that two distinct packets $P$ and $P'$ contain the same tuple $(n_1,\ldots,n_r)$, so that when computing $c(P)$ and $c(P')$ we would carry out the multiplication $\epsilon_1^{n_1} \cdots \epsilon_r^{n_r}$ twice. Similarly, a tuple from the list $U_0$ may appear in some packet $P$. By doing step (11) we are thus avoiding unnecessary recalculations and reducing arithmetic with fundamental units.
\item Before carrying out step (10) it may prove useful to order the elements of $U$ according to the corresponding height $h_{\boldsymbol u}$. Even though this has an additional cost, it could prevent the unnecessary checking of many inequalities $h_{\boldsymbol u}\leq w$ in step (10): traversing the list $U$, one would check this inequality only until it fails once, and then no further elements of $U$ need to be considered.
\item In order to reduce both the cost of arithmetic operations in $K$ and the amount of memory used in carrying out steps (11) - (13) of Algorithm \ref{main2}, compact representations of algebraic integers could be used (see \cite{thiel}); this would be especially useful for storing and working with fundamental units. All the operations on algebraic numbers that are needed in Algorithm \ref{main2} can be done efficiently working only with compact representations, thus avoiding the need to store and compute with the numbers themselves.
\end{itemize}
\end{rems}

Note that several quantities appearing in Algorithm \ref{main2} involve real numbers, so that an implementation of the algorithm may require floating point arithmetic. For some applications this may not be an issue, but if one needs to know with certainty that all elements not exceeding the specified height bound have been found, then it is crucial to choose the precision for floating point calculations carefully. In the next section we will address this problem in detail.

\section{Error analysis}

There are two issues that must be considered in order to implement Algorithm \ref{main2} in such a way that the output is guaranteed to be complete and correct. These issues are due to the fact that in a computer we cannot work exactly with the real numbers that appear in the algorithm (heights of algebraic numbers, logarithms of real numbers, absolute values of algebraic numbers), so we must make do with rational approximations of them. We consider now the question of finding approximations that are good enough to guarantee correct results.

The first issue is that of computing the height of an algebraic number. In carrying out Algorithm \ref{main2} one must check inequalities of the form $h_K(x)\leq D$ for given $x\in K^{\ast}$ and $D\in\R$. In practice, one can only work with  rational approximations $\tilde h$ of $h_K(x)$ and $\tilde D$ of $D$, and check whether $\tilde h\leq\tilde D$. However, it may happen that $h_K(x)\leq D$ even though $\tilde h>\tilde D$. To deal with this problem one must be able to find arbitrarily close rational approximations of $h_K(x)$.

The second issue is that of enumerating lattice points inside a polytope, which is needed in Algorithm \ref{units1}. The polytopes considered are of the form $T(\B)$, where $\B=[-d,d]^r$ is a box in $\R^r$ and $T:\R^r\to\R^r$ is a linear isomorphism. In practice, the box $\B$ must be replaced by a box $\tilde\B$ with rational vertices, and the matrix of $T$ must be approximated by a rational matrix corresponding to a map $\tilde T$. We will not necessarily have an equality $\Z^r\cap T(\B)=\Z^r\cap \tilde T(\tilde\B)$, so lattice points may be lost in this approximation process. One must therefore take care to ensure that good enough approximations are found so that at least there is a containment $\Z^r\cap T(\B)\subseteq\Z^r\cap \tilde T(\tilde\B)$.

There are several ways of dealing with these issues, each one leading to a different implementation of the main algorithm. For concreteness, we describe in this section one way of solving these problems, and we give the corresponding modification of Algorithm \ref{main2}.

We introduce the following terminology to be used throughout this section: if $\vec x = (x_1,\ldots,x_m)$ is a vector in the Euclidean space $\R^m$, we say that $\vec y = (y_1,\ldots,y_m) \in \R^m$ is a \emph{$\delta$-approximation} of $\vec x$ if $|x_i - y_i| < \delta$ for all $1 \le i \le m$. 

\subsection{The height function} Given an element $x\in K^{\ast}$ and a real number $\lambda>0$, we wish to compute a rational number $\tilde h$ such that $|\tilde h-h_K(x)|<\lambda$. Writing $x=\alpha/\beta$ with $\alpha,\beta\in\o_K$ and using \eqref{height_via_log}, we see that $h_K(x)$ can be approximated by first finding good approximations of the vectors $\Lambda(\alpha)$ and $\Lambda(\beta)$.

\begin{lem}\label{height_delta}
Fix $\lambda > 0$ and set $\delta = \lambda/(r+2)$. Let $\alpha$, $\beta$ be nonzero elements of $\o_K$. Let $\tilde n$, $(s_1,\ldots,s_{r+1})$, and $(t_1,\ldots,t_{r+1})$ be $\delta$-approximations of $\log(N(\alpha,\beta))$, $\Lambda(\alpha)$, and $\Lambda(\beta)$, respectively. Then, with
	\[ \tilde h := -\tilde n + \sum_{i=1}^{r+1} \max\{s_i,t_i\} \]
we have $|h_K(\alpha/\beta) - \tilde{h}| < \lambda$.
\end{lem}

\begin{proof}
Let $\Lambda(\alpha)=(x_1,\ldots, x_{r+1})$ and $\Lambda(\beta)=(y_1,\ldots, y_{r+1})$. Using \eqref{height_via_log} we obtain
	\begin{align*}
		|h_K(\alpha/\beta) - \tilde{h}| \le |\tilde n - \log(N(\alpha,\beta))| + \sum_{i=1}^{r+1} \left|\max\{x_i,y_i\} - \max\{s_i,t_i\} \right|
		< (r+2)\delta = \lambda.
	\end{align*}
\end{proof}

For a nonzero element $y\in K$, each entry of the vector $\Lambda(y)$ is of the form $n_v \log |y|_v$ for some place $v\in M_K^{\infty}$. Corresponding to $v$ there is an embedding $\sigma:K\hookrightarrow\C$ such that $|y|_v=|\sigma(y)|$. Since $\sigma(y)$ is a complex root of the minimal polynomial of $y$, known methods (see \cite{pan}, for instance) can be applied to approximate $\sigma(y)$ with any given accuracy. In this way the vector $\Lambda(y)$ can be approximated to any required precision.

Lemma \ref{height_delta} provides a way of approximating the height of any element of $K$ by using the map $\Lambda$. However, in practice a slightly different method will be needed for computing heights of units. As mentioned in \S\ref{improvements}, in order to avoid costly arithmetic with fundamental units we do not work directly with units $u$ when carrying out Algorithm \ref{main2}. Hence, we cannot approximate the vector $\Lambda(u)$ by computing $|u|_v$ for every place $v$. Instead, a unit $u=\epsilon_1^{n_1} \cdots \epsilon_r^{n_r}$ is encoded by the tuple $(n_1,\ldots, n_r)$, so we need a way of approximating $h_K(u)$ given only this tuple. Since $\Lambda(u)=\sum_{j=1}^rn_j\Lambda(\epsilon_j)$, it is enough to approximate the vectors $\Lambda(\epsilon_i)$ sufficiently well; we make this precise in the following lemma.

\begin{lem}\label{unit_height_delta}
Fix $\lambda, M > 0$ and set $\delta = \lambda/(r(r+1)M)$. Let $\{\epsilon_1,\ldots,\epsilon_r\}$ be a system of fundamental units for $\o_K^{\times}$, and for each $j$ let $(s_{1,j},\ldots,s_{r+1,j})$ be a rational $\delta$-approximation of $\Lambda(\epsilon_j) = (x_{1,j},\ldots,x_{r+1,j})$. Suppose $u = \epsilon_1^{n_1} \cdots \epsilon_r^{n_r}$ is a unit with $|n_1|,\ldots,|n_r| \le M$. Then, with
	\[ \tilde h := \sum_{i=1}^{r+1} \max\left\{\sum_{j=1}^r n_js_{i,j}\;,0\right\} \]
we have 
	\[ |h_K(u) - \tilde h| < \lambda. \]
\end{lem}

\begin{proof}
Since $\Lambda(u) = \sum_{j=1}^r n_j\Lambda(\epsilon_j)$, the $i$-th coordinate of $\Lambda(u)$ is given by $\sum_{j=1}^r n_jx_{i,j}$. Applying \eqref{height_via_log} to $u = u/1$ yields
	\[ h_K(u) = \sum_{i=1}^{r+1} \max\left\{\sum_{j=1}^r n_jx_{i,j}\;, 0\right\}. \]
Therefore,
	\begin{align*}
		|h_K(u) - \tilde h| = \left|\sum_{i=1}^{r+1} \left( \max\left\{\sum_{j=1}^r n_jx_{i,j} , 0\right\} - \max\left\{\sum_{j=1}^r n_js_{i,j} , 0\right\} \right) \right|
		\le \sum_{i=1}^{r+1} \sum_{j=1}^r |n_j| \cdot |x_{i,j} - s_{i,j}|
		< r(r+1)M\delta
		= \lambda. 
	\end{align*}
\end{proof}

\subsection{Units of bounded height} We use here the notation from $\S 3.2$. Let $d=\log D$ and let $\B=[-d,d]^r$. Algorithm \ref{units1} requires that we enumerate all integer lattice points in the polytope $S^{-1}(\B)$. In practice, the matrix $S$ must be replaced by a rational approximation $\tilde S$, and the box $\B$ by a rational box $\tilde\B$. We show here how to choose these approximations so that $S^{-1}(\B)\subseteq \tilde S^{-1}(\tilde\B)$. For the purpose of enumerating integer lattice points, we may then replace $\B$ with $\tilde\B$ and $S$ with $\tilde S$, thus avoiding errors arising from floating-point arithmetic.

For a vector $v\in\R^r$ we denote by $|v|$ the usual Euclidean norm of $v$, and for a linear map $L:\R^r\to\R^r$ we let $\|L\|$ denote the operator norm, \[\|L\|=\sup_{|x|\leq 1}|Lx|.\] We also denote by $L$ the matrix of $L$ with respect to the standard basis for $\R^r$. Recall that the supremum norm of $L$ is given by $\|L\|_{\text{sup}}:=\max_{i,j}|L_{i,j}|$, and that there is an inequality
\begin{equation}\label{op_norm_ineq} \|L\|\le r\sqrt{r}\cdot\|L\|_{\text{sup}}.\end{equation}

We begin with two results which will be useful for approximating the inverse of a matrix.

\begin{lem}\label{rudin_ineq} Let $V$ be an $r\times r$ invertible matrix over the real numbers, and let $\tilde V$ be a matrix such that \[\|\tilde V-V\|\cdot\|V^{-1}\|<1.\] Then $\tilde V$ is invertible and \[\|\tilde V^{-1}-V^{-1}\|\leq\frac{\|\tilde V-V\|\cdot\|V^{-1}\|^2}{1-\|\tilde V-V\|\cdot\|V^{-1}\|}.\]
\end{lem}

\begin{proof} See the proof of  Theorem 9.8 in \cite{rudin}. 
\end{proof}

Using \eqref{op_norm_ineq} we obtain:

\begin{cor}\label{inverse_approx} With $V$ as in the lemma, let $m$ be a constant with $m\geq r^2\cdot\|V^{-1}\|_{\mathrm{sup}}$. Given $\lambda>0$, let $\tilde V$ be a matrix such that $\|\tilde V-V\|_{\mathrm{sup}}<\frac{\lambda}{r^2(m^2+m\lambda)}$. Then $\tilde V$ is invertible and $\|\tilde V^{-1}-V^{-1}\|<\lambda$.
\end{cor}

We can now give the required accuracy in approximating the matrix $S$.

\begin{prop}\label{S_approx} Let $S$ be an invertible $r\times r$ matrix over the real numbers, and let $d$ be a positive real number. Given $\eta>0$, define $\tilde\B=[-d-\eta, d+\eta]^r$. Let $m$ be a real number such that \[m\geq r^2\cdot\max\{\|S\|_{\mathrm{sup}},\|S^{-1}\|_{\mathrm{sup}}\}.\] Define constants \[\lambda:=\frac{\eta}{dr(1+m)}\;\; and \;\; \delta:=\min\left\{\frac{\lambda}{r^2(m^2+m\lambda)},\frac{1}{r^2}\right\}.\]  

If $\tilde S$ is any $r\times r$ matrix such that $\|\tilde S-S\|_{\mathrm{sup}}<\delta,$ then $\tilde S$ is invertible and $S^{-1}(\B)\subseteq\tilde S^{-1}(\tilde\B)$. 
\end{prop}

\begin{proof} It follows from Corollary \ref{inverse_approx} that $\tilde S$ is invertible and $\|\tilde S^{-1}-S^{-1}\|<\lambda$. For any $x\in\B$ we then have \[|\tilde S(S^{-1}x)-x| = |\tilde S(S^{-1}x) - \tilde S(\tilde S^{-1}x)| \le \|\tilde S\|\cdot \|S^{-1}-\tilde S^{-1}\|\cdot |x|<\eta.\] Hence, we see that $\tilde S(S^{-1}x) \in\tilde \B$, so $S^{-1}x\in \tilde S^{-1}(\tilde \B)$ and this completes the proof.
\end{proof}

Proposition \ref{S_approx} reduces the problem of finding an adequate approximation of $S$ to finding upper bounds for $\|S\|_{\mathrm{sup}}$ and $\|S^{-1}\|_{\mathrm{sup}}$. The former can be easily done using any approximation of $S$. One way of finding an upper bound for $\|S^{-1}\|_{\mathrm{sup}}$ is to use the fact that $S^{-1}=\frac{1}{\det(S)}\cdot A$, where $A$ is the adjugate matrix of $S$. By approximating $S$ one can obtain a lower bound for $\det(S)$ and an upper bound for the entries of $A$.

\subsection{Revised algorithm}\label{revised} Using the methods described above we give a new version of Algorithm \ref{main2} which takes precision issues into account. We assume here that $r>0$; the case $r=0$ is treated in \S\ref{bdd_height_quad} for imaginary quadratic fields, and is trivial when $K=\Q$.

\begin{alg}[Algebraic numbers of bounded height]\label{main3}\mbox{}\\
Input: A number field $K$, a bound $B\geq 1$, and a tolerance $\theta \in (0,1]$, with $B,\theta\in\Q$.\\
Output: Two lists, $L$ and $L'$, such that:
\begin{itemize}
\item If $x\in K$ satisfies $H_K(x)\leq B$, then $x$ is in either $L$ or $L'$.
\item For every $x\in L$, $H_K(x)< B$.
\item For every $x\in L'$,  $|H_K(x)-B|<\theta$.
\end{itemize}

\smallskip
\begin{enumerate}[itemsep = 1.1mm]
\item Set $t = \theta/(3B)$ and let $\delta_1 = t/(6r+12)$. Find a complete set $\{\a_1,\ldots,\a_h\}$ of ideal class representatives for $\o_K$, and for each index $\l$ compute a rational $\delta_1$-approximation of $\log(N(\a_{\l}))$.
\item Construct the set \[\mathcal N:=\bigcup_{\l=1}^h\{m\cdot N(\a_{\l}):m\le B\}\] and make a list $\mathcal{P}$ of all nonzero principal ideals of $\o_K$, each represented by a single generator $g$, having norm in $\mathcal N$. For each $g$, find a $\delta_1$-approximation of $\Lambda(g)$.
\item For each ideal $\a_{\l}$, make a list $(g_{\l,1}),\ldots, (g_{\l,s_{\l}})$ of all elements of $\mathcal{P}$ contained in $\a_{\l}$ whose norms are at most $B\cdot N(\a_{\l})$.
\item For each index $\l$:

\begin{enumerate}
\item Make a list $R_{\l}$ of pairs $(i,j)$ such that $1\leq i < j \leq s_{\l}$ and $(g_{\l,i},g_{\l,j})=\a_{\l}$.
\item For each pair $(i,j)$ in $R_{\l}$:
\begin{itemize}[itemindent = -5mm]
\item[] Use Lemma \ref{height_delta} and data from steps (1) and (2) to find a rational approximation of $h_K(g_{\l,i}/g_{\l,j})$.
\item[] The result will be a rational number $r_{\l,i,j}$ such that $|r_{\l,i,j}-h_K(g_{\l,i}/g_{\l,j})|<t/6$.
\end{itemize}
\end{enumerate}

\item Find a rational number $b$ such that $\frac{t}{12}<b-\log(B)<\frac{t}{4}$ and set $\tilde d = b + \frac{t}{6} + \max_{\l}\max_{(i,j)\in R_{\l}} r_{\l,i,j}$. 
\item Compute a system of fundamental units $\boldsymbol\epsilon=\{\epsilon_1,\ldots,\epsilon_r\}$ and find a constant $m$ such that \[m\geq r^2\cdot\max\{\|S(\boldsymbol\epsilon)\|_{\mathrm{sup}},\|S(\boldsymbol\epsilon)^{-1}\|_{\mathrm{sup}}\}.\]
\item Define constants \[\tilde\lambda=\frac{t/12}{\tilde dr(1+m)}, \;\tilde\delta=\min\left\{\frac{\tilde\lambda}{r^2(m^2+m\tilde\lambda)},\frac{1}{r^2}\right\}, \;M=\lceil\tilde d(m+\tilde\lambda \sqrt r)\rceil,\; \delta_2=\min\left\{\tilde\delta,\frac{t/6}{r(r+1)M}\right\}.\]

\item Compute $\delta_2$-approximations $v_1,\ldots, v_r$ of the vectors $\Lambda(\epsilon_1),\ldots, \Lambda(\epsilon_r)$, construct the $r\times r$ matrix $\tilde S$ whose $j$-th column is the vector $v_j$ with its last coordinate deleted, and compute $\tilde S^{-1}$.
\item Construct a list $U$ consisting of all integer vectors $(n_1,\ldots,n_r)$ in the polytope $\tilde S^{-1}([-\tilde{d}, \tilde d]^r)$.
\item Create a list $L$ containing only the element 0, and create empty lists $U_0, U_0'$ and $L_0,L_0'$.
\item For each tuple $\boldsymbol u=(n_1,\ldots,n_r)$ in $U$:
	\begin{enumerate}
	\item Compute the vector $\tilde\Lambda_{\boldsymbol u}=\sum_{j=1}^rn_jv_j$.
	\item Using $\tilde\Lambda_{\boldsymbol u}$ and Lemma \ref{unit_height_delta}, find a rational approximation of $h_K\left(\epsilon_1^{n_1} \cdots \epsilon_r^{n_r}\right)$. The result will be a rational number $r_{\boldsymbol u}$ such that  $|r_{\boldsymbol u}-h_K\left(\epsilon_1^{n_1} \cdots \epsilon_r^{n_r}\right)|<t/6.$
	\item If $r_{\boldsymbol u} < b-\frac{5}{12}t$, then append $\boldsymbol u$ to $U_0$.
	\item If $b-\frac{5}{12}t\le r_{\boldsymbol u}< b+\frac{1}{12}t$, then append $\boldsymbol u$ to $U_0'$.
	\item If $r_{\boldsymbol u}>t/12+\tilde d$, then remove $\boldsymbol u$ from $U$.
	\end{enumerate}
\item For each index $\l$\;: 
\begin{itemize}[itemindent = -2mm]
\item[] For each pair $(i,j)\in R_{\l}$\;: 
\begin{itemize}[itemindent = -6mm]
\item[] Let $w=b + r_{\l,i,j}+\frac{1}{4}t$.
\item[] For each tuple $\boldsymbol u=(n_1,\ldots,n_r)$ in $U$:
\begin{itemize}[itemindent = -8mm]
\item[] If $r_{\boldsymbol u}< w$, then:

\begin{enumerate}[itemindent = -1mm]
\item Let $P$ be the packet $\left(\l,(i,j),(n_1,\ldots,n_r)\right)$.
\item Use the data from steps (1), (2), and 11(a), together with \eqref{height_via_log}, to find a rational approximation of $h_K(c(P))$. The result will be a rational number $r_P$ with $|r_P-h_K(c(P))|<t/3$.
\item If $r_P\le b-\frac{7}{12}t$, then append the packet $P$ to $L_0$.
\item If $b-\frac{7}{12}t<r_P< b+\frac{1}{4}t$, then append the packet $P$ to $L_0'$.
\end{enumerate}
\end{itemize}
\end{itemize}
\end{itemize}
\item Make a list consisting of the distinct tuples $(n_1,\ldots, n_r)$ appearing in $U_0,U_0'$ or in some packet $P$ in $L_0$ or $L_0'$. Compute and store all units of the form $\epsilon_1^{n_1} \cdots \epsilon_r^{n_r}$ with $(n_1,\ldots, n_r)$ in this list.
\item For each tuple $(n_1,\ldots,n_r)$ in $U_0$, append to $L$ all numbers of the form $\zeta \epsilon_1^{n_1} \cdots \epsilon_r^{n_r}$ with $\zeta\in\mu_K$, and similarly for $U_0'$ and $L'$. Use the units computed in step (13).
\item For each packet $P$ in $L_0$, append to $L$ all numbers of the form $\zeta\cdot c(P)$ and $\zeta/c(P)$ with $\zeta\in\mu_K$, and similarly for $L_0'$ and $L'$. Use the data from step (13) when computing the numbers $c(P)$.
\item Return the lists $L$ and $L'$.
\end{enumerate}
\end{alg}

We make the following comments regarding various steps of Algorithm \ref{main3}:

\begin{itemize}[itemsep = 1.1mm]
\item Let $S=S(\boldsymbol\epsilon)$. With $d$ as in step (5) of Algorithm \ref{main2} we have $\tilde d>d+t/12$. Therefore, if we set $\eta = t/12$ and let $\lambda$ and $\delta$ be defined as in Proposition \ref{S_approx}, then $\tilde \lambda < \lambda$ and $\tilde \delta \le \delta$. By construction, $\|\tilde S-S\|_{\mathrm{sup}}<\tilde\delta \le \delta$, so that by Proposition \ref{S_approx} we have
	\[ S^{-1}([-d,d]^r)\subseteq \tilde S^{-1}([-d-\eta,d+\eta]^r) \subseteq \tilde S^{-1}([-\tilde d, \tilde d]^r).\]
\item In order to use Lemma \ref{unit_height_delta} in step 11(b) we must know that $|n_i|\leq M$ for all $1 \le i \le r$. Since $\boldsymbol u\in\tilde S^{-1}([-\tilde{d}, \tilde d]^r)$, we clearly have the upper bound $|n_i|\leq \tilde d\sqrt r\|\tilde S^{-1}\|$. By Corollary \ref{inverse_approx}, $\|\tilde S^{-1}-S^{-1}\|<\tilde\lambda$, so applying \eqref{op_norm_ineq} we have $\|\tilde S^{-1}\|\leq r\sqrt r\|S^{-1}\|_{\mathrm{sup}}+\tilde\lambda$. It follows that $|n_i|\leq \tilde d(m+\tilde\lambda\sqrt r) \le M$.
\item The condition $r_{\boldsymbol u} + \frac{5}{12}t < b$ from step 11(c) implies that $h_{\boldsymbol u}<\log B$; the condition $b - \frac{5}{12}t \le r_{\boldsymbol u} < b + \frac{1}{12}t$ from step 11(d) implies $|h_{\boldsymbol u}-\log B|<t$. Moreover, every $\boldsymbol u = (n_1,\ldots,n_r)$ for which $h_K(\epsilon_1^{n_1} \cdots \epsilon_r^{n_r}) \le \log B$ is in $U_0$ or $U_0'$, since $h_{\boldsymbol u} \le \log B$ implies $r_{\boldsymbol u} < b + \frac{1}{12}t$.
\item The condition $r_{\boldsymbol u} - t/12 > \tilde d$ from step 11(e) implies that $h_{\boldsymbol u}>d$.
\item The condition $r_P + \frac{7}{12}t \le b$ from step 12(c) implies that $h_K(c(P))<\log B$; the condition $b - \frac{7}{12}t < r_P < b + \frac{1}{4}t$ from step 12(d) implies that $|h_K(c(P))-\log B|<t$. Moreover, every packet $P$ with $h(c(P))\le\log B$ is in either $L_0$ or $L_0'$, since $h_K(c(P))\leq \log B$ implies that $r_P<b+\frac{t}{4}$.
\item Elements $x \in L$ satisfy $H_K(x) < B$, since they come from tuples in $U_0$ and packets in $L_0$. Elements $x \in L'$ satisfy $|h_K(x) - \log B| < t$, which implies that $|H_K(x)-B|<\theta$ by the Mean Value Theorem.
\end{itemize}

Note that the list $L'$ of Algorithm \ref{main3} consists of elements $x\in K$ whose heights are so close to $B$ that it is not possible to decide whether $H_K(x)\leq B$ with the tolerance specified as input. In particular, $L'$ might contain elements of height exactly $B$. For general number fields $K$ we cannot prevent this from occurring; however, for quadratic fields we can prevent it, as explained below.

\subsection{Case of quadratic fields}\label{bdd_height_quad}

We give in this section a way to shorten the list $L'$ from Algorithm \ref{main3} in the case of real quadratic fields, and to eliminate it altogether in the case of imaginary quadratic fields.

\begin{prop} \label{quads} Let $K$ be a quadratic field and let $x\in K^{\ast}$. Let $\sigma$ be the generator of $\Gal(K/\Q)$.

\begin{enumerate}[itemsep = 1.1mm] 
\item If $K$ is an imaginary field, then $H_K(x)$ is an integer.
\item If $K$ is a real field, then $H_K(x)\in\Q$ if and only if $\max\{|x|,|\sigma(x)|\}\leq1$ or $\min\{|x|,|\sigma(x)|\}\geq 1$. Moreover, if $H_K(x)\in\Q$, then $H_K(x)\in\Z$.
\end{enumerate}

\end{prop}

\begin{proof} Write $x=a/b$ with $a,b\in \o_K$, and let $\a=(a,b)$ be the ideal generated by $a$ and $b$ in $\o_K$. Then $H_K(x)=N(\a)^{-1}\prod_{v\in M_K^{\infty}}\max\{|a|_v^{n_v},|b|_v^{n_v}\}$. There are coprime ideals $I$ and $J$ of $\o_K$ such that $(a)=\a\cdot I$ and $(b)=\a\cdot J$. We then have $H_K(x)=N(J)\prod_{v\in M_K^{\infty}}\max\{|x|_v^{n_v},1\}$.
\noindent 
\begin{enumerate}[itemsep = 1.1mm] 
\item If $K$ is an imaginary field, then $H_K(x)=N(\a)^{-1}\max\{N_{K/\Q}(a),N_{K/\Q}(b)\}=\max\{N(I),N(J)\}\in\Z.$
\item If $K$ is a real field, then $H_K(x)=N(J)\max\{|x|,1\}\cdot \max\{|\sigma(x)|,1\}$. If $\max\{|x|,|\sigma(x)|\}\leq1$, then $H_K(x)=N(J)\in\Z$. If $\min\{|x|,|\sigma(x)|\}\geq 1$, then \[H_K(x)=N(J)|N_{K/\Q}(x)|=N(J)N(a)/N(b)=N(I)\in\Z.\] Now suppose that $\max\{|x|,|\sigma(x)|\}> 1$ and $\min\{|x|,|\sigma(x)|\}<1$. Then, without loss of generality we may assume that $|x|<1<|\sigma(x)|$. It follows that $x\notin\Q$, so $H_K(x)=N(J)|\sigma(x)|\notin\Q$.
\end{enumerate}
\end{proof}

In the case of real quadratic fields it is possible to detect some elements of the list $L'$ from Algorithm \ref{main3} which should be in the list $L$: if $x\in L'$ has height $H_K(x)\in\Q$ --- a condition which can be determined using Proposition \ref{quads} ---  then by construction of $L'$ it must be the case that $H_K(x)$ is the unique integer closest to $B$ (assuming the tolerance $\theta$ from Algorithm \ref{main3} was chosen to be less than $1/2$). If the nearest integer is $\lfloor B\rfloor$, then $x$ can then be deleted from $L'$ and appended to $L$. Otherwise, $x$ can be deleted from $L'$. At the end of this process the list $L'$ will only contain elements of $K$ whose heights are irrational and very close to $B$.

For imaginary quadratic fields $K$ there is a modification of Algorithm \ref{main2} that allows us to determine elements of bounded height without doing any height computations. Thus, for such fields we avoid the need for a list $L'$ as in Algorithm \ref{main3}.

With the notation and terminology of \textsection\ref{main_thm_section} we have:

\begin{thm}\label{iq_thm} Let $K$ be an imaginary quadratic field. Then \[\{\gamma\in K^{\ast}:H_K(\gamma)\leq B\}=\mu_K\cup\bigcup_{B\text{-packets\;} P} F(P).\]

\end{thm}

\begin{proof} One containment follows from Theorem \ref{main_thm}, since $\o_K^{\times}=\mu_K$ in this case. It is therefore enough to show that $H_K(c(P))\leq B$ for every packet $P$. Letting $P=(\l,(i,j))$ we have \[N(\a_{\l}) H_K(c(P))=\prod_{v\in M_K^{\infty}}\max\{|g_{\l,i}|_v^{n_v},|g_{\l,j}|_v^{n_v}\}=\max\{N_{K/\Q}(g_{\l,i}),N_{K/\Q}(g_{\l,j})\}\leq B\cdot N(\a_{\l}).\] Hence, $H_K(c(P))\leq B$.
\end{proof}

Theorem \ref{iq_thm} leads to the following algorithm.

\begin{alg}[Numbers of bounded height in an imaginary quadratic field]\label{iq}\mbox{}\\
Input: An imaginary quadratic field $K$ and a bound $B\geq 1$.\\
Output: A list of all elements $x\in K$ satisfying $H_K(x)\leq B$.
\begin{enumerate}[itemsep = 1.1mm]
\item Create a list $L$ containing 0 and all elements of $\mu_K$.
\item Find a complete set $\{\a_1,\ldots,\a_h\}$ of ideal class representatives for $\o_K$.
\item Construct the set \[\mathcal N:=\bigcup_{\l=1}^h\{m\cdot N(\a_{\l}):m\le B\}\] and make a list $\mathcal{P}$ of all nonzero principal ideals of $\o_K$, each represented by a single generator $g$, having norm in $\mathcal N$.
\item For each ideal $\a_{\l}$, make a list $(g_{\l,1}),\ldots, (g_{\l,s_{\l}})$ of all elements of $\mathcal{P}$ contained in $\a_{\l}$ whose norms are at most $B\cdot N(\a_{\l})$.
\item For each index $\l$:
\begin{itemize}[itemindent = -2mm]
\item[] For each pair of indices $(i,j)$ such that $1\leq i < j \leq s_{\l}$ and $(g_{\l,i},g_{\l,j})=\a_{\l}$:
\begin{itemize}[itemindent = -6mm]
\item[] Let $c=g_{\l,i}/g_{\l,j}$ and append to $L$ all elements of the form $\zeta \cdot c$ and $\zeta/c$ with $\zeta\in\mu_K$.
\end{itemize}
\end{itemize}
\item Return the list $L$.
\end{enumerate}
\end{alg}

Note that, by Theorem \ref{main_thm}, the list $L$ will not contain duplicate elements.

\section{Efficiency of the algorithm}\label{efficiency}

We discuss in this section a measure of the efficiency of Algorithm \ref{main2} --- henceforth abbreviated A3 --- and of the algorithm of Peth\H{o} and Schmitt --- abbreviated PS --- proposed in \cite{petho/schmitt}. Given a number field $K$ and a height bound $B$, both methods begin by computing some basic data attached to $K$: in the case of PS we need an integral basis for $\o_K$, and in the case of A3 we need also the ideal class group and a set of fundamental units. After this step, both methods construct a set of elements of $K$ which is known to contain the desired set of numbers of bounded height; this larger set will be called the {\it search space} of the method and denoted by $\mathcal S_{\mathrm{PS}}(B)$ or $\mathcal S_{\mathrm A3}(B)$. Once a search space is known, the two methods proceed to compute the height of each element in this set and check whether it is smaller than $B$. We will measure the efficiency of a method by comparing the size of the search space to the size of the set of elements of height $\le B$. Thus, we define the {\it search ratio} of PS to be the number \[\sigma_{\mathrm{PS}}(B) := \frac{\#\mathcal S_{\mathrm{PS}}(B)}{\#\{x\in K:H_K(x)\leq B\}}\;,\] 
and similarly for A3.

We begin by stating the main result on which PS is based. To do this we will need the {\it Minkowski embedding} $\Phi:K\hookrightarrow\R^n$ defined by 
\begin{equation}\label{minkowski_embedding} \Phi(x) = \left(\sigma_1(x),\ldots,\sigma_{r_1}(x),\sqrt 2\Re\tau_1(x),\sqrt 2\Im\tau_1(x),\ldots,\sqrt 2\Re\tau_{r_2}(x),\sqrt 2\Im\tau_{r_2}(x)\right).\end{equation}

Recall that $\Phi(\o_K)\subset\R^n$ is a lattice of rank $n$ and determinant $|\Delta_K|^{1/2}$. In PS one requires a reduced basis for this lattice, which can be obtained by applying the LLL algorithm to any given basis. 

\begin{thm}[Peth\H{o}, Schmitt]\label{psthm}
Let $K$ be a number field of degree $n$ over $\Q$, and let $B\ge 1$ be a real number. Let $\{\omega_1,\ldots, \omega_n\}$ be an LLL-reduced integral basis for $\o_K$. Then every element $x\in K$ with $H_K(x)\leq B$ can be written in the form \[x=\frac{a_1\omega_1+\cdots+a_n\omega_n}{c},\] where $a_1,\ldots, a_n$ and $c$ are integers satisfying \[1\leq c\leq B\hspace{0.5cm}\mathrm{and}\hspace{0.5cm} |a_i|\leq 2^{n(n-1)/4}Bc\;.\]
\end{thm}

\begin{proof} By \cite[Thm. 1]{petho/schmitt} it suffices to show that for every index $i$ we have 
\[\sqrt{n}\cdot\prod_{j\ne i}\|\Phi(\omega_j)\|\le 2^{n(n-1)/4}|\Delta_K|^{1/2}.\]
This inequality is a consequence of the fact that the given basis is reduced (see (1.8) in \cite{lll}) together with the fact --- noted in the proof of \cite[Thm. 2]{petho/schmitt} --- that $\|\Phi(\omega_j)\|\ge\sqrt n$ for every $j$.
\end{proof}

\begin{rem} The bound for $|a_i|$ given in Theorem \ref{psthm} is slightly better than the one proved in \cite[Thm. 2]{petho/schmitt}; the reason for this is that the Minkowski embedding used in \cite{petho/schmitt} does not include the factors of $\sqrt 2$ for the complex embeddings of $K$.
\end{rem}

Theorem \ref{psthm} leads to the following algorithm:

\begin{alg}[PS]\label{ps}\mbox{}\\
Input: A number field $K$ and a bound $B\geq 1$.\\
Output: A list of all elements $x\in K$ satisfying $H_K(x)\leq B$.
\begin{enumerate}[itemsep = 1.1mm]
\item Compute an LLL-reduced integral basis $\omega_1,\ldots, \omega_n$ for $\o_K$.
\item Create an empty list $L$.
\item For $c=1$ to $\lfloor B\rfloor$\;: 
\begin{enumerate}
\item Let $D=\lfloor 2^{n(n-1)/4 }Bc\rfloor$\;.
\item For every integer tuple $(a_1,\ldots, a_n)\in [-D,D]^n\;$:

\smallskip
\begin{enumerate}
\item Let $\displaystyle x=\frac{a_1\omega_1+\cdots+a_n\omega_n}{c}$\;.
\item If $H_K(x)\leq B$, then append $x$ to $L$.
\end{enumerate}
\end{enumerate}
\item Return the list $L$.
\end{enumerate}
\end{alg}

\medskip We now give our main result comparing the efficiency of PS with that of A3. 
 
\begin{thm}\label{eff} Let $K$ be a number field of degree $n$. The search ratios of {\rm PS} and {\rm A3} satisfy \[\sigma_{\mathrm{PS}}(B)\gg B^{2n-2}\;\;\;\;\mathrm{and}\;\;\;\;\sigma_{\mathrm{A3}}(B) \ll (\log B)^r,\] 

where $r$ is the rank of the unit group of $\o_K$.
\end{thm}

\begin{proof}
By Schanuel's formula \cite{schanuel} we know that there is a constant $C_K$ such that \[\#\{P\in K:H_K(P)\leq B\}\sim C_KB^2.\] A simple calculation shows that the size of the search space in PS satisfies $\#\mathcal S_{\mathrm{PS}}(B)>B^{2n}2^{n^2(n-1)/4+n}$; the first statement in the theorem then follows easily. Now let $\boldsymbol{\a}=\{\a_1,\ldots,\a_h\}$ be a complete set of ideal class representatives for $\o_K$ and $\boldsymbol\epsilon=\{\epsilon_1,\ldots,\epsilon_r\}$ a system of fundamental units. For each index $\l$ let $g_{\l,1},\ldots, g_{\l,s_{\l}}$ be generators for all nonzero principal ideals contained in $\a_{\l}$ whose norms are at most $B\cdot N(\a_{\l})$. By \cite[\S 2.5.4]{bor/sha} we may assume that \begin{equation}\label{g_bound}|g_{\l,i}|_v^{n_v}\leq E_K(B\cdot N(\a_{\l}))^{1/(r+1)}\end{equation} for every place $v\in M_{K}^{\infty}$ and all indices $\l,i$. Here $E_K$ is a constant which depends on $\boldsymbol\epsilon$ but not on $B$. 

Let $P(B)=\sum_{\l=1}^h s_{\l}$. Using the bound given in \cite[Thm. 1]{murty} we find that  $P(B)\ll B$. Using Theorem \ref{main_thm} we see that the size of the search space considered in A3 satisfies \begin{equation}\label{cand_bound}\#\mathcal S_{\mathrm{A3}}(B)\le 1+ \#\{u\in\o_K^{\times}:H_K(u)\leq D\}\cdot\left(1+2\cdot P(B)^2\right)\ll B^2\cdot\#\{u\in\o_K^{\times}:H_K(u)\leq D\},\end{equation} 

where $D$ is any number such that $D\geq B\cdot\max_{\l,i,j}H_K(g_{\l,i}/g_{\l,j})$.  By (\ref{g_bound}) we have \[H_K(g_{\l,i}/g_{\l,j})\leq \frac{\prod_{v\in M_K^{\infty}}\max\{|g_{\l,i}|_v^{n_v},|g_{\l,j}|_v^{n_v}\}}{N(\a_{\l})}\leq \frac{\prod_{v\in M_K^{\infty}}E_K(B\cdot N(\a_{\l}))^{1/(r+1)}}{N(\a_{\l})}=F_KB\] for all $\l,i,j$ and for some constant $F_K$ independent of $B$. Hence, we may take $D=F_KB^2$. By Corollary \ref{unit_number}, \[\#\{u\in\o_K^{\times}:H_K(u)\leq D\}\ll(\log B)^r.\] Therefore, by (\ref{cand_bound}), the size of the search space in A3 satisfies $\#\mathcal S_{\mathrm{A3}}(B)\ll B^2(\log B)^r.$ The second statement in the theorem follows from this inequality and Schanuel's asymptotic estimate.
\end{proof}

\section{Computational complexity and performance}\label{comp_perf} Theorem \ref{eff} shows that, for a fixed field $K$, the method A3 is asymptotically (as $B\to\infty$) much more efficient than PS. However, the search ratio is not the only factor determining total computation time: for instance, in A3 the initial step of computing the ideal class group and unit group of $\o_K$ can be very time-consuming if the discriminant of $K$ is large. We will briefly discuss here the time complexity of various steps in PS and A3. In order to illustrate the theoretical statements made here, we will also give examples based on explicit computations done with the two methods. Both algorithms have been implemented in Sage \cite{sage}, and all computations below have been done on a Mac Pro with a Quad-Core 2.26 GHz processor and 8 GB of memory. A precision must be chosen for floating point calculations, so we will fix a precision of 100 bits in all examples. 

The first step in PS is to determine an integral basis for $\o_K$. It is known by work of Chistov \cite[Thms. 2 and 3]{chistov} that this is polynomial-time equivalent to determining the squarefree part of the discriminant of a defining polynomial $P$ for $K$; in other words, the first problem can be solved in polynomial time if and only if the second one can. Unfortunately, finding the squarefree part of an integer is not much easier than factoring it, and  current methods for finding integral bases (see, for instance, \cite[\S6.1]{cohen}) do indeed begin by factoring disc$(P)$. Assuming the general number field sieve \cite{nf_sieve} is used for this factorization, the expected running time will be \[O\left(\exp((64d/9)^{1/3}(\log d)^{2/3})\right),\] where $d=1+\lfloor \log_2 \mathrm{disc}(P)\rfloor$ is the bit length of disc$(P)$. Once an integral basis $\boldsymbol\omega=\{\omega_1,\ldots, \omega_n\}$ is known, the LLL algorithm can be applied to the image of $\boldsymbol\omega$ under the Minkowski embedding \eqref{minkowski_embedding} to obtain a reduced basis for $\o_K$. This step will run in time \[O\left(n^6(\log \max_{1\le i\le n}\|\Phi(\omega_i)\|)^3\right).\] After computing a reduced basis for $\o_K$, the remaining time in PS is spent computing the heights of all elements in a search space, and it is here that PS has its main drawback. As mentioned in the proof of  Theorem \ref{eff}, the size of the search space is greater than $B^{2n}2^{n^2(n-1)/4+n}$ --- much larger than the size of the set PS is attempting to compute, which is roughly $B^2$. To illustrate the effect this large search space has on computation time, we give in Table \ref{PSrange} below the time required by PS to compute elements of bounded height in three number fields. We have chosen here the quadratic, cubic, and quartic fields of smallest discriminant (such data may be found in Jones's number field database \cite{jones}), so that the time required to compute an integral basis for $K$ is negligible. In the table, a field is specified by giving a defining polynomial for it. Note that, even with the very small bounds $B$ chosen here, the computing times with PS are far from optimal: for comparison, the same computations using A3 took 0.03 seconds for the quartic field, 0.02 seconds for the cubic field, and 0.03 seconds for the quadratic field. This great difference in computing times for A3 and PS is due to the extremely large search space used in PS.

\begin{table}[ht]
\centering
\begin{tabular}{|c|c|c|c|}
\hline
Number field $K$ & Height bound $B$ & PS time & $\sigma_{\mathrm{PS}}(B)$\\
\hline
$x^4 - x^3 + x^2 - x + 1$ & $2$ & 36.15 hours & $1.73\times 10^6$\\
\hline
$x^3 - x^2 -2x+ 1$ & 4 &  3.15 hours & 77,175\\
\hline
$x^2 - x + 1$ & $20$ &  $5.85$ hours & 12,630\\
\hline
\end{tabular}

\bigskip
\caption{Sample computations with PS}
\label{PSrange}
\end{table}

Given these times, we do not consider PS to be a practical method to use when $K$ has degree larger than 2, and even for quadratic fields the height bound $B$ must be very small in order for PS to terminate in a reasonable time.  

We discuss now the complexity of A3. In this algorithm, the initial step can have a substantial cost: in addition to a reduced integral basis as required by PS, in A3 we need to compute ideal class representatives and fundamental units for $\o_K$; by \cite[Thm. 5.5]{lenstra}, this step can be done in time
\[(2+\log|\Delta_K|)^{O(n)}|\Delta_K|^{3/4}.\]
This appears to be the best known upper bound on the running time for a deterministic computation of class groups and unit groups. However, assuming the generalized Riemann hypothesis (GRH), there are subexponential probabilistic algorithms: for instance, Buchmann \cite{buch} gives an algorithm which has expected running time 
\[\exp\left((1.7+ o(1))\sqrt{\log|\Delta_K|\log\log|\Delta_K|}\right).\] 
If $K$ is a field for which the cost of this initial step is high, then PS may perform better than A3, especially if the height bound $B$ is small. As an example of this phenomenon, consider the field $K=\Q(\sqrt{2928239983})$, which has class number $1,472$. Using PS to find all elements of height at most 4, the time is $16.19$ seconds, while using A3 and assuming GRH the time is 34.34 seconds. However, for the height bound $B=20$, the time with PS increases to 9 hours and 40 minutes, while the time for A3 (still assuming GRH) only increases to 36.37 seconds.

In step (2) of A3 one has to determine elements of given norm in $\o_K$. A good method for doing this is provided in \cite{fincke/pohst}; however, the authors do not give a precise statement of the complexity of their method. In any case, the cost of using this algorithm is tied to the size of the fundamental units computed in step (1). If the units are very large, then this step can be costly and even dominate the computation time of the entire algorithm. Consider, for instance the field $K=\Q(\sqrt{123456789123})$. A fundamental unit for $K$ has the form $a+b\sqrt{123456789123}$, where $a$ and $b$ are both greater than $10^{2096}$. We apply A3 to this field with height bound $B=3,000$, and we assume GRH in order to quickly obtain the class group and unit group of $\o_K$. The total computation time is 3.7 minutes, with $88\%$ of the time spent on step (2).

Steps (4), (9), and (10) of A3 are mostly spent on height computations. The cost of a single height computation depends on the precision used for floating point arithmetic, and the number of required height computations is determined by the size of the search space of A3. Since we have already studied the size of this space in the previous section, we will not discuss here the cost of these three steps. Note, however, that it will certainly be much smaller than the time spent on height computations in PS, since the search space is substantially smaller.

In step (7) of A3 one must find all integer lattice points that lie inside a polytope in $\R^r$ given by a collection of $2^r$ rational vertices. The polytope is obtained by applying the linear map $S^{-1}$ to a box of the form $[-d,d]^r$.  Here, $S$ is the $r\times r$ matrix whose columns are the vectors $\tilde\Lambda(\epsilon_j)$.  The integer points inside this polytope can be determined once a generating function for the polytope has been computed; by \cite[Thm. 4.4]{barvinok/pommersheim}, this can be done in time $\mathcal{L}^{O(r)}$, where $\mathcal{L}$ is the sum of the bit lengths of the coordinates of all vertices of the polytope. If the embeddings $K\hookrightarrow\C$ are computed with precision $p$, then the vertices of the polytope consist of floating point numbers of precision $p$, and thus we have $\mathcal L=2^r\cdot rp$. Hence, the generating function can be computed in time 
\[(2^r\cdot rp)^{O(r)}.\]

As suggested by the above complexity, an increase in the rank $r$ can lead to a significant increase in the time required for this step. For example, consider the following two sextic fields: 
\[K_1: x^6+2 \mathrm{\;\;\;and\;\;\;} K_2: x^6 - 6x^4 + 9x^2 - 3.\]
The fields $K_1$ and $K_2$ have the same degree, discriminants of similar size (namely, $-1492992$ and 1259712), equal class numbers (namely 1), and both have small fundamental units; however, the unit group rank is 2 for  $K_1$ and 5 for $K_2$. This means that in step (7) of the algorithm, in the case of $K_1$ we must compute integer points inside a polytope in $\R^2$, while in the case of $K_2$ we must work in $\R^5$. Applying A3 over both fields with height bound $B=500$, we find that for $K_1$ the time spent on step (7) is 0.005 seconds, while for $K_2$ it is 21.6 minutes.

In steps (11) - (13) of A3, the previously computed data of integer lattice points and packets is used to construct the final output list of numbers. This process consists entirely of arithmetic operations with elements of $K$, and its computational cost is largely determined by the size of the fundamental units computed in step (1). These units can be very large: for instance, one can reasonably expect that there are infinitely many real quadratic fields $K$ for which a fundamental unit cannot be written down with fewer than $|\Delta_K|^{1/2}$ bits (see \cite[\S 5]{lenstra}). If the chosen fundamental units are large, then the cost of these final steps can be considerable; however, we remark that in practice the effect of large fundamental units on computation time is much greater in step (2) than in these final steps. For example, in the computation mentioned above for the field $K=\Q(\sqrt{123456789123})$, only $4\%$ of the time was spent on steps (11) - (13).     

\subsection{Sample computations with A3}\label{performance} We end by giving a series of examples showing that A3 and Algorithm 5 (abbreviated A5) can be applied with number fields of various degrees and with several different height bounds which would be far beyond the practical range of applicability of PS. Note that, by Theorem \ref{iq_thm}, the search ratio of A5 is always 1.

\bigskip
\begin{table}[!htb]
\centering
\begin{tabular}{|c|c|c|c|}
\hline
$B$ & A5 time &  $\sigma_{\mathrm{A5}}(B)$ & Elements found\\
\hline
200 & 0.85 seconds & 1 & 15,275\\
\hline
1,000 &  18.46 seconds & 1 &  393,775\\
\hline
5,000 &   7.61 minutes & 1 & 9,761,079 \\
\hline
\end{tabular}

\medskip
\caption{Computations with A5 over the field $K=\Q(\sqrt{-107})$}
\label{A5_example}
\end{table} 

\begin{table}[!ht]
\centering
\begin{tabular}{|c|c|c|c|}
\hline
$B$ & A3 time &  $\sigma_{\mathrm{A3}}(B)$ & Elements found\\
\hline
$200$ &  5.69 seconds & 14.49 & 2,143\\
\hline
$1,000$ &  34.71 seconds & 17.07 & 54,679\\
\hline
$5,000$ & 7.40 minutes & 20.05 & 1,365,315\\
\hline

\end{tabular}

\bigskip
\caption{Computations with A3 over the field $K=\Q(\sqrt{36865})$}
\label{both_rq}
\end{table}

\begin{table}[!htb]
\centering
\begin{tabular}{|c|c|c|c|}
\hline
$B$ & A3 time & $\sigma_{\mathrm{A3}}(B)$ & Elements found\\
\hline
100 & 4.76 seconds & 88.66 &  5,123\\
\hline
500 & 1.68 minutes & 149.78 &  124,911\\
\hline
1,000 &  7.42 minutes & 175.62 & 489,255 \\
\hline
\end{tabular}

\medskip
\caption{Computations with A3 over the field $K: x^6+2$}
\label{A3_sextic_example}
\end{table}

\begin{table}[!htb]
\centering
\begin{tabular}{|c|c|c|c|}
\hline
$B$ & A3 time & $\sigma_{\mathrm{A3}}(B)$ & Elements found\\
\hline
100 & 3.27 minutes & 28,807 &  2,679\\
\hline
500 & 1.33 hours & 23,635 &  81,251\\
\hline
1,000 & 11.14 hours & 52,553 & 316,915 \\
\hline
\end{tabular}

\medskip
\caption{Computations with A3 over the cyclotomic field $\Q(\zeta_{13})$}
\label{A3_12_example}
\end{table}

\printbibliography[title=References]

\bigskip
\end{document}